\documentclass[11pt, oneside]{article}   	
\usepackage{geometry, amsmath,tikz,amsthm}                		
\usepackage{graphicx}				
\usepackage{amssymb}
\usepackage{enumitem}

\newcommand\partitionfr[1]{
	\coordinate (prev) at (0,0);
	\foreach \dir in {#1}{
		\draw[help lines, line width = .25mm] (prev) -- +(0,1) coordinate (prev);
		\draw[help lines, line width = .25mm] (prev)+(0,-1) grid +(\dir,0);
	};
}

\newcommand\dyckpathshade[3]{
	start point, size, Dyck word (size x 2 booleans)
	\fill[blue]  (#1) rectangle +(#2,#2);
	\fill[white]
	(#1)
	\foreach \dir in {#3}{
		\ifnum\dir=0
		-- ++(1,0)
		\else
		-- ++(0,1)
		\fi
	} |- (#1);
}
\newcommand\dyckpath[3]{

	\draw[help lines] (#1) grid +(#2,#2);
	\draw[dashed] (#1) -- +(#2,#2);
	\coordinate (prev) at (#1);
	\foreach \dir in {#3}{
		\ifnum\dir=0
		\coordinate (dep) at (1,0);
		\else
		\coordinate (dep) at (0,1);
		\fi
		\draw[line width=4pt,] (prev) -- ++(dep) coordinate (prev);
	};
}
\newcommand\dyckpathmn[4]{
	\draw[help lines] (#1) grid +(#2,#3);
	\coordinate (prev) at (#1);
	\foreach \dir in {#4}{
		\ifnum\dir=0
		\coordinate (dep) at (1,0);
		\else
		\coordinate (dep) at (0,1);
		\fi
		\draw[line width=4pt,black] (prev) -- ++(dep) coordinate (prev);
	};
}

\DeclareMathOperator{\SYT}{SYT}
\DeclareMathOperator{\bH}{{\bf H}}

\DeclareMathOperator{\comaj}{{comaj}}
\DeclareMathOperator{\Des}{{Des}}
\DeclareMathOperator{\des}{{des}}

\DeclareMathOperator{\Comaj}{{Comaj}}
\DeclareMathOperator{\SSYT}{{SSYT}}
\DeclareMathOperator{\Exp}{{Exp}}
\DeclareMathOperator{\bF}{{\bf F}}

\DeclareMathOperator{\frob}{\mathcal{F}}

\newtheorem{theorem}{Theorem}[section]
\newtheorem{proposition}[theorem]{Proposition}

\newtheorem{lemma}[theorem]{Lemma}
\newtheorem{corollary}[theorem]{Corollary}

\title{Kronecker powers of harmonics, polynomial rings, and generalized principal evaluations}
\author{Marino Romero}

\begin{document} 
\maketitle
\begin{abstract}
Our main goal is to compute the decomposition of arbitrary Kronecker powers of the Harmonics of $S_n$. To do this, we give a new way
 of decomposing the character for the action of $S_n$ on polynomial rings with $k$ sets of $n$ variables. There are two aspects to this decomposition. The first is algebraic, in which formulas can be given for certain restrictions from $GL_n$ to $S_n$ occurring in Schur-Weyl duality. 
The second is combinatorial. We give a generalization of the $\comaj$ statistic on permutations which includes the $\comaj$ statistic on standard tableaux. This statistic allows us to write a generalized principal evaluation for Schur functions and Gessel Fundamental quasisymmetric functions. 
\end{abstract}

\section{Introduction}
 
Let $\bH_n = \mathbb{C}[x_1,\dots,x_n]/(p_1,\dots,p_n)$, where $p_i = x_1^i+\cdots + x_n^i$ is the $i^{\text{th}}$ power sum symmetric polynomial.  The power sum polynomials generate all invariants $I_n = \mathbb{C}[x_1,\dots, x_n]^{S_n}$ of the symmetric group \cite{Weyl}, and $\bH_n$ is referred to as the coinvariants of $S_n$. This quotient is isomorphic to the harmonics of the symmetric group, which as a set equals
$$
\left\{  P(\partial_{x_1},\dots, \partial_{x_n}) V_n : ~ P \in \mathbb{C}[x_1,\dots,x_n]     \right\}, 
$$
the linear span of derivatives of the Vandermonde determinant
$$
V_n = \prod_{i<j } x_i -x_j.
$$

Our main goal will be to compute the graded multiplicities of irreducible representations appearing in $\bH_n^{\otimes k}$.  By Chevalley's work \cite{Chevalley}, 
we have that the power sum polynomials are algebraically independent and $ \mathbb{C}[x_1,\dots,x_n] = \bH_n \otimes  I_n$. Therefore
$$
 \mathbb{C}[x_1,\dots,x_n]^{\otimes k} = \bH_n^{\otimes k}  \otimes I_n^{\otimes k}.
$$
In order to decompose $\bH_n^{\otimes k}$ into irreducible representations, we give a new way of computing the multigraded multiplicities in $\mathbb{C}[x_1,\dots,x_n]^{\otimes k}$, which we view as the polynomial ring $\mathbb{C}[X^k_n]$ with $k$ sets of variables
$$
X^k_n = \{ {}_1x_1,\dots,  {}_1x_n, {}_2x_1,\dots, {}_2x_n, \dots, {}_kx_1,\dots, {}_kx_n\}.
$$
Orellana and Zabrocki \cite{OrellanaZabrocki} have recently done this (with the inclusion of anti-commuting variables) by viewing 
$$
\mathbb{C}[X^k_n] \simeq  \bigoplus_{\ell (\lambda) \leq \min(k,n)} W^{\lambda}_k \otimes W^{\lambda}_n
$$
 under the duality from the action of $GL_k \times GL_n$ \cite{GoodmanWallach}. Restricting to $S_n$, they are then able to compute the multiplicites of irreducible representations. 
\\

  Any $S_n$-module $M$ with character $\chi^M$ can be decomposed into irreducible representations, giving
  $$
  M= \bigoplus_{\lambda \vdash n} A_{\lambda}^{\oplus n_\lambda},
  $$
  where $n_\lambda$ is the multiplicity of the irreducible representation $A_\lambda$. 
  The connection between representations and symmetric functions in the variables $X = x_1+x_2 + \cdots $ can be encapsulated by the following identity:
 $$
 \bF(\chi^M) = \frac{1}{n!} \sum_{\sigma \in S_n } \chi^M(\sigma) p_{\lambda(\sigma)}[X] = \sum_{\lambda \vdash n} n_\lambda s_\lambda[X].
  $$
  In our notation, $\lambda(\sigma)$ is the partition giving the cycle type of $\sigma$, and 
  $$
s_\lambda[X] = \sum_{T\in SSYT(\lambda)} x^T  
  $$
is a Schur function. The plethystic evaluation $s_{\lambda}[x_1+x_2+\cdots ]$ is equivalent to the ordinary evaluation $s_{\lambda}(x_1,x_2,\dots).$
 
 Given a multigraded $S_n$-module $$M= \bigoplus_{a_1,a_2,\dots} M_{a_1,a_2,\dots},$$ the Schur expansion of the Frobenius characteristic
\begin{align*}
 \mathcal{F} M & = \sum_{a_1,a_2,\dots} \bF( \chi^{M_{a_1,a_2,\dots}}) q_1^{a_1} q_2^{a_2}\cdots  \\ & = \frac{1}{n!} \sum_{\sigma \in S_n}  p_{\lambda(\sigma)}[X] \sum_{a_1,a_2,\dots} \chi^{M_{a_1,a_2,\dots} }(\sigma) q_1^{a_1} q_2^{a_2} \cdots.
\end{align*}
 will give the multigraded multiplicities of irreducible representations in $M$. In the following section, we will do this for the ring of polynomials with $k$ sets of $n$ variables.\\

Our main goal will be to show that
\begin{theorem} \label{KroneckerHarmonics}
For any $n$, $k$ and $\lambda \vdash n$, we have
$$
\langle \mathcal{F} \bH_n^{\otimes k}, s_\lambda \rangle = \sum_{T \in \SYT(\lambda)} \sum_{\vec{\sigma} \in S_{n}^{\times k-1} }q^{\comaj_T(\vec{\sigma},\epsilon)}.
$$ 
\end{theorem}
The sums run over all standard tableaux $T$ of shape $\lambda$ and sequences $\vec{\sigma} = (\sigma^1,\dots,\sigma^{k-1}) $ of permutations $\sigma^i \in S_n$; moreover, we use $\epsilon$ to denote the identity element of $S_n$.
 This new statistic $\comaj_T$ generalizes $\comaj$ in the following ways: \begin{enumerate} \item
 When $T$ is a row of size $n$ and $\sigma \in S_n$,  
 $$\comaj_{T}(\sigma) = \comaj(\sigma) = \sum_{i \in \des(\sigma) } (n-i),
 $$
 where 
 $$
 \des(\sigma) = \{ i: \sigma_{i} > \sigma_{i+1} \}.
 $$
 For example, in this element of $S_9$, we have marked the positions of descents with a bullet point. 
 $$
\sigma = 4 {\overset{\bullet}{7}} 3 { \overset{\bullet}{9}} 1 2 5{ \overset{\bullet}{8}}6
 $$
 This means $\des(\sigma) = \{ 2,4,8\}$. To compute $\comaj(\sigma)$, we can think of labeling each $\sigma_i$, from left to right. We start by labeling $\sigma_1$ with $0$. After labeling $\sigma_i$, we move to $\sigma_{i+1}$. We increase our labelling by $1$ if $\sigma_i$ is marked as a descent. Otherwise, we keep the same label used on $\sigma_i$. The above example would then become
  $$
0 {\overset{\bullet}{0}} 1 { \overset{\bullet}{1}} 2 2 2{ \overset{\bullet}{2}}3
 $$
 It will be more relevant to instead fill the position $\sigma_i$ by the label we have constructed:
 $$
2 21 0 2 3{{0}} { {2}} {{1}}  
 $$
 We can then uniquely determine the permutation from this labeling.
 Summing the labels together gives $\comaj(\sigma) = 2 (1) + 4 (2) + 1 (3) = 13.$ 
 
 As a reference for these classical statistics, we have Stanley's book \cite{Stanley}.
 \\ \item
Let $\epsilon$ be the identity element in $S_n$. Then for any standard tableau $T \in \SYT(\lambda)$ we have
$$
\comaj_T(\epsilon) = \comaj (T) = \sum_{ i \in \des(T)} (n-i),
$$
where 
$$
\des(T) = \{i: \text{ $i+1$ is above or to the left of $i$ in $T$ } \}.
$$
 For instance, here we have a standard tableau of shape $(4,2,1)$.
 \begin{align*} 
\begin{tikzpicture}[scale= .5]
\partitionfr{4,2,1}\draw (-2,1) node {$T~=$};
\draw (.5,.5) node {$1$};
\draw (1.5,.5) node {$ 2 $};
\draw ( .5,1.5) node {$3 $};
\draw (2.5,.5) node {$4 $};
\draw (3.5,.5) node {$5$};
\draw (1.5,1.5) node {$6$};
\draw (.5,2.5) node {$7$};
\end{tikzpicture}
\end{align*}
Then 
$$
\des(T) = \{2,5,6\}.
$$
Using the same labeling strategy as above, we can replace $1$ by $0$. Then after relabeling $i$, we relabel $i+1$. If $ i$ is a descent, we increase the label by $1$. Otherwise, we use the same label as $i$. For the above example, we get
 \begin{align*} 
\begin{tikzpicture}[scale= .5]
\partitionfr{4,2,1};
\draw (.5,.5) node {$0$};
\draw (1.5,.5) node {$ 0$};
\draw ( .5,1.5) node {$1$};
\draw (2.5,.5) node {$1 $};
\draw (3.5,.5) node {$1$};
\draw (1.5,1.5) node {$2$};
\draw (.5,2.5) node {$3$};
\end{tikzpicture}
\end{align*}

Then $\comaj(T) = 11$ is the sum of the entries in this tableau. 
\end{enumerate}
 
 The proof of Theorem \ref{KroneckerHarmonics} relies on the following generalized principal evaluation for Schur functions. To simplify the expressions, we will use the P{o}chhammer symbol:
 $$(q;t)_r = (1-q)(1-qt)\cdots (1-qt^{r-1}).$$
 \begin{theorem}\label{finiteevaluation}
For $\lambda \vdash n$ and any $k$, we have
\begin{align*}
s_{\lambda}\left[ \frac{1}{(1-q_1)\cdots(1-q_k)} \right] = \frac{ \sum_{T\in \SYT(\lambda)} \sum_{\vec{\sigma} \in S_{n}^{\times k-1 } } q_1^{\comaj^1_T (\vec{\sigma})} \cdots q_k^{\comaj^k_T(\vec{\sigma}) }}{ (q_1;q_1)_n \cdots (q_k;q_k)_n}.
\end{align*}
 \end{theorem}
 This plethystic evaluation is equivalent to evaluating $s_\lambda$ at all monomials in the variables $q_1,\dots,q_k$.
 This theorem gives a new way of writing the multigraded multiplicity of the irreducible representation corresponding to $\lambda$ in the polynomial ring with $k$ sets of variables.
 Orellana and Zabrocki \cite{OrellanaZabrocki} have a combinatorial interpretation for this expression in terms of monomial symmetric functions in the $q_i$. It can also be viewed as a character on $GL_k$ resulting from the duality.
 
 The proof of our theorem relies on Proposition \ref{generalproposition}, which interprets the product
 $$
(q_1;q_1)_n \cdots (q_k;q_k)_n s_{\lambda}\left[ \frac{1}{(1-q_1)\cdots(1-q_k)} \right] 
 $$
 as a sequence of injections on semistandard tableaux. Taking the compliments of the images, we arrive at a set of fixed points, given precisely by a list of $k-1$ permutations.
 The proposition can also be applied to infinite sequences, giving the infinite case of the principal evaluation as well.
  Let $Q = q_1+q_2+ \cdots$, and let 
 $$
 \Exp[Q]= \sum_{i \geq 0} h_i [Q] = \exp \sum_{i>0} p_k[X]/k
 $$
 be the generating series for the homogeneous symmetric functions. 
 \begin{theorem}\label{infiniteevaluation}
 For $\lambda \vdash n$ we have
\begin{align*}
s_{\lambda} \left[ \Exp[Q] \right] & =   s_\lambda\left[  \frac{1}{ (1-q_1)(1-q_2) \cdots}   \right] \\[.5em]  &= \frac{\sum_{T \in \SYT(\lambda)} \sum_{\vec{\sigma} \in_T S_{n}^{\infty} }Q^{\Comaj_T(\vec{\sigma})}}{  \prod_{i \geq 1}   (q_i;q_i)_n}. 
\end{align*} 
\end{theorem}
We should note that Loehr and Warrington \cite{LoehrWarrington} have given an expression for $s_\lambda[s_\mu]$ in terms of quasisymmetric functions. A useful reference for us is Loehr and Remmel's overview of plethystic calculus \cite{LoehrRemmel}, where they formally describe how this expression can be expanded in terms of semistandard tableaux or quasisymmetric functions.

We will see that our result actually stems from a statement regarding a generalized principal evaluation on Gessel's fundamental basis for quasisymmetric functions \cite{Gessel}. This is given in Theorem \ref{quasitheorem}. Related work and expansions can be found in Gessel's survey of $P$-partitions \cite{GesselSurvey}. 
Recent work of Nadeau and Tewari \cite{NadeauTewari} give a principal specialization with one variable, motivated by the study of Schubert calculus. There is no obvious relation between their results and the present paper, but it would be interesting to see if or how they relate.

In the following section, we will show why Theorem \ref{finiteevaluation} implies Theorem \ref{KroneckerHarmonics}. The subsequent section will define the statistics involved in the main theorems. In particular, this section alone will be sufficient for understanding the statement in Theorem \ref{finiteevaluation}. 
We will then prove the finite and infinite generalized principal evaluations, and then finish by looking at some specific cases.\\

Loehr and Remmel's work \cite{LoehrRemmel} and Macdonald's book \cite{Macdonald} are good references for the symmetric and quasisymmetric function expressions we use here. Another useful reference with symmetric functions and permutation statistics is Mendes and Remmel's book \cite{MendesRemmel}.
For some representation theory background, one may find Goodman and Wallach's book \cite{GoodmanWallach} and Sagan's book \cite{Sagan} useful. For a nice treatment on the representation theory of the symmetric group, we refer to Garsia and E\u{g}ecio\u{g}lu's book \cite{GarsiaEgecioglu}.

  \section{Polynomial rings with $k$ sets of $n$ variables}
  
In this section we will look at the action of $S_n$ on $\mathbb{C}[X^k_n].$  We recall that when we have a single set of variables, the power sum polynomials $p_1,\dots,p_n$ are algebraically independent and generate $I_n = \mathbb{C}[x_1,\dots,x_n]^{S_n}$. We have 
 $$
 \mathcal{F} \mathbb{C}[x_1,\dots, x_n]  =  \frac{\mathcal{F} \bH_n }{(1-q)\cdots (1-q^n)}.
 $$
Using only a $q$ to record total degree, we can then write
 $$
 \frob \bH_n^{\otimes k} = (q;q)_n^k \frob \mathbb{C}[x_1,\dots, x_n]^{\otimes k} .
 $$

We use the variables $X^k_n=\{ {}_1x_1,\dots,  {}_1x_n, {}_2x_1,\dots, {}_2x_n, \dots, {}_kx_1,\dots, {}_kx_n\}$ to write the module on the right-hand side as 
 $$
 \mathbb{C}[x_1,\dots, x_n]^{\otimes k} \simeq \mathbb{C}[X^k_n].
 $$
The monomial basis is given by elements of the form
$$
x^{\alpha} = \prod_{i=1}^k\prod_{j=1}^n ( {}_ix_j)^{\alpha_j^i},
$$
 where $\alpha = (\alpha^1_1,\dots, \alpha^1_n,\dots, \alpha^k_1,\dots,\alpha^k_n)$ is a sequence of natural numbers $\alpha^i_j \in \mathbb{N} = \{0,1\dots\}$.
 Equivalently, we can set $\beta^r= (\alpha^1_r,\dots, \alpha^k_r)$ and denote a basis element by the list of sequences 
$$\beta=( \beta^1,\dots, \beta^n) ~~~ {\text{  with  } }~~~ \beta^i = (\alpha^1_i,\dots, \alpha^k_i). $$ We then also set $x^\beta =x^\alpha$. The orbit of this monomial, under the action of $S_n$, is given by permuting the elements of $\beta$. That is,
$$
\sigma x^\beta = x^{\sigma \beta} ~~~ \text{  with  } ~~~ \sigma \beta = (\beta^{\sigma_1}, \dots, \beta^{\sigma_n}).
$$ 

Suppose we have the permutation $\sigma = (1, 2, 3, 4) (5,6,7)(8)$ of cycle type $\lambda$. Then the sum of monomials $x^{\beta}$ which are fixed by $\sigma$ can be generated by the power series
$$
\prod_{i=1}^k  \frac{1}{1-{}_ix_1 ~{}_ix_2~ {}_ix_3 ~{}_ix_4} \cdot \frac{1}{1- {}_ix_5~ {}_ix_6 ~{}_ix_7} \cdot \frac{1}{1- {}_ix_8}
$$

We will now count the degree of ${}_ix_j$ using the variable $q_i$. Then the graded character of the action of $\sigma$ on $\mathbb{C}[X^k_n]$ is given by the series 
$$
\prod_{i=1}^k  \frac{1}{1-q_i^4} \frac{1}{1-q_i^3} \frac{1}{1-q_i} = \prod_{i=1}^k p_{4,3,1}\left[\frac{1}{1-q_i } \right] = p_\lambda\left[ \frac{1}{(1-q_1) \cdots (1-q_k)} \right]. 
$$
This holds for any permutation, meaning that the Frobenius image of the multi-graded character of the action of $S_n$ on $\mathbb{C}[X^k_n]$ is given by
\begin{align*}
R_{n,k}  = &~~~ \sum_{\lambda \vdash n} p_\lambda\left[ \frac{1}{(1-q_1)\cdots (1-q_k)} \right] p_\lambda[X]/z_\lambda&& \\
  =& ~~~h_n\left[ \frac{X}{(1-q_1)\cdots (1-q_k) } \right]  &&\\
  =& ~~~\sum_{\lambda \vdash n} s_{\lambda}\left[ \frac{1}{(1-q_1) \cdots (1-q_k)} \right] s_\lambda[X].&  &
\end{align*}
The above equalities are instances of Cauchy's identity, where for any two homogeneous dual bases $\{u_\lambda\}_\lambda, \{v_\lambda\}_\lambda$ and any two expressions $X$ and $Y$, we have
$$
h_n[XY] = \sum_{\lambda \vdash n} u_\lambda[X] v_\lambda[Y].
$$
We have then shown, by setting all $q_i = q$, that
$$
\langle \mathcal{F} \bH_n^{\otimes k} , s_\lambda \rangle = \left( \prod_{i=1}^{k} (q_i;q_i)^k_n \right) \langle R_{n,k} , s_\lambda \rangle \Big|_{q_i = q} = (q;q)^k_ns_\lambda\left[ \frac{1}{(1-q)^k} \right].
$$
Therefore, Theorem \ref{KroneckerHarmonics} is a consequence of Theorem \ref{finiteevaluation}. The following section will give a description of the combinatorial formula in Theorem \ref{finiteevaluation}.

\section{The $\Comaj$ statistic}
We are going to introduce a general $\comaj$ statistic on $S_n$ that depends on two parameters. One is any subset $R \subseteq \{1,\dots, n-1\} $ and the other is any list of sequences $S = (s^1,\dots, s^n) \in (\mathbb{N}^r)^n$. In this notation, $s^i= (s^i_1,\dots,s^i_r) \in \mathbb{N}^r.$  To help present our examples below, we will sometimes omit the commas in $s^i$ and write $s^i=s^i_1\cdots s^i_r$.

An effort will be made to denote sequences with a superscript and numbers with subscripts.
We will write $(z_1,\dots, z_r) < (s_1,\dots, s_r)$ for two sequences if at the first index of disagreement we have $z_i < s_i$; moreover, $z -1 < s $ will mean  $z \leq s$. This is the lexicographic ordering on sequences.

We are going to call two indices, $i<j$, $R$-neighbors if $i,i+1,\dots, j-1 \in R$. We will indicate that two indices $i$ and $j$ are $R$-neighbors by writing $i \sim_R j$.
Define
\begin{align*}
\Des_{R,S}(\sigma) = \{i :  s^{\sigma_{i}} >  s^{\sigma_{i+1}} 
&- \chi(   \text{ $\sigma_{i+1} < \sigma_i $ and $\sigma_{i+1} \not\sim_R\sigma_{i}    $} )  \\ 
&- \chi( \text{ $\sigma_{i+1} > \sigma_i $ and $\sigma_i \sim_R\sigma_{i+1}  $})               \}.
\end{align*}
We are using the indicator function $\chi(A)$, which gives $1$ if $A$ is true, and $0$ otherwise. If $r = 0$ so that $S$ is empty, we use $s^i = 0$ for all $i$ to compare inequalities.
\\

The idea of a descent is easier to explain by example. Before that, we note that if $s^{\sigma_{i } } > s^{\sigma_{i+1}}$, then we always have a descent at $i$. If $s^{\sigma_{i}} = s^{\sigma_{i+1}}$, then we have a descent only when one of the following holds:
\begin{itemize}
\item   $\sigma_{i+1} < \sigma_i $ and $\sigma_{i+1} \not \sim_R\sigma_{i}    $, or
\item   $\sigma_{i+1} > \sigma_i $ and $\sigma_i \sim_R\sigma_{i+1}  $.
\end{itemize}

Now for our example, we consider $R= \{ 3,4,6  \}$,  $\sigma = 1456723$, and 
$$S = ( 020  , 312  ,  312 , 011   ,   011, 100  ,010   ).$$
The $R$-neighbors are given by $3 \sim_R 4\sim_R 5$ and $6 \sim_R 7$.
We first see that the inequality conditions on the $s^i$ (ordered by $\sigma$) give 
$$
\overset{\bullet}{s}^1 > \overset{\bullet}{s}^4 = s^5 < \overset{\bullet}{s}^6>s^7 <s^2= s^3.
$$
We have labeled with a bullet mark the location of all the descents, given by $\{1,2,4\}$. Note that there is a descent at $2$ because $s^{\sigma_2} = s^4 = s^{\sigma_{3}}= s^5$, and $4 \in R$ (or $3 \sim_R 4$). However, even though $s^2 = s^3$, we do not have a descent at $6$ because $2 \notin R$ (or $2 \not \sim_R 3$).
\\
Define 
$$\Comaj_{R,S}(\sigma) = \sum_{i \in \Des_{R,S}(\sigma)} (n-i).
$$
Furthermore, define for $R$, $S \in (\mathbb{N}^r)^n$ and $\sigma$,
$$
Z_{R,\sigma}(S) = (z^1,\dots, z^n) \in (\mathbb{N}^{r+1})^n
$$
with
\begin{itemize}
\item $(z^i_2,\dots, z^i_n) = s^i$, and
\item $z^{\sigma_{i}}_1 = \sum_{j < i} \chi(j \in \Des_{R,S}(\sigma))$.
\end{itemize}
In other words, start with $z^{\sigma_1}_1 = 0$. If we have set $z^{\sigma_i}_1 = c$, then we set $z^{\sigma_{i+1}}_1 = c+1$ if $i \in \Des_{R,S}(\sigma)$, and set $z^{\sigma_{i+1}}_1 = c$ otherwise.  To see this in the previous example, which has $\sigma = 1456723$ and descent set $\{1,2,4\}$, we have
$$Z_{\{3,4,6\},\sigma}(S) = ( 0020  , 3312  ,  3312 , 1011   ,   2011, 2100  , 3010   ).$$
We see that that if $Z = Z_{R,\sigma}(S) $, then
$$
\Comaj_{R,S}(\sigma) = z^1_1+ \cdots + z^n_1
$$
is the sum of the first coordinates in the resulting list of sequences.
Define for any set $R$ and any sequence of permutations $\vec{\sigma}= (\sigma^1,\sigma^2,\dots)$,
$$
Z^r_{R, \vec{\sigma}}= Z_{R,\sigma^r} \cdots Z_{R,\sigma^1} ( \emptyset).
$$
For $\vec{\sigma} \in S_n^{\times k-1}$ and $T \in \SYT(\lambda)$, let
$
Z^r = Z^r_{\des(T), \vec{\sigma}}
$ and define for $1 \leq i \leq k$,
$$
\comaj^i_T(\vec{\sigma}) = \Comaj_{\des(T),Z^{i-1}}(\sigma^i)
$$
with $\sigma^k = \epsilon.$ Then Theorem \ref{finiteevaluation}
states that 
$$
\left(\prod_{i=1}^k (q_i;q_i)_n \right) s_{\lambda}\left[ \frac{1}{(1-q_1)\cdots (1-q_k)} \right] = \sum_{T \in \SYT(\lambda)} \sum_{\vec{\sigma} \in S_n^{\times k-1}}  q_1^{\comaj^1_T(\vec{\sigma}) } \cdots q_k^{\comaj^k_T(\vec{\sigma})}
$$

To fully illustrate the statistic, we will work out an example. Suppose we are given
$$
(3651274, 6523417, 1423567) \in S_7^{ \times 3}
$$ 
and 
\begin{align*} 
\begin{tikzpicture}[scale= .5]
\partitionfr{4,2,1}\draw (-2,1) node {$T~=$};
\draw (.5,.5) node {$1$};
\draw (1.5,.5) node {$ 2 $};
\draw ( .5,1.5) node {$3 $};
\draw (2.5,.5) node {$4 $};
\draw (3.5,.5) node {$5$};
\draw (1.5,1.5) node {$6$};
\draw (.5,2.5) node {$7$};
\end{tikzpicture}
\end{align*}
Then 
$$
\des(T) = \{2,5,6\}.
$$
Let $R = \{2,5,6\}$, meaning that the $R$-neighbors are given by $2 \sim_R 3$ and $5 \sim_R 6 \sim_R 7$.
We start with the permutation $\sigma^1= 365 127 4$. 
We mark the descents of $\sigma$ depending on $R,\emptyset$ with a bullet:
$$
\sigma^1 =  36 \overset{\bullet}{5}   12 \overset{\bullet}{7}  4
$$
Note that $2$ is not a descent because $\sigma_2 = \sigma_3 +1$ and $\sigma_3= 5 \in \des(T)$ (so that $\sigma_2 \sim_R \sigma_3$). 
Replacing position $i$ by the number of descents which preceded $\sigma_i$,
we get
$$
Z^1 = (1,1,0,2,0,0,1).
$$
We now look at $\sigma^2 = 6523417$.
From here we can compute that 
$$
\Des_{R,((1),(1),(0),(2),(0),(0),(1))} \left(6523417\right)   = \{ 3, 5 \}.
$$
The descents indicate that $z^6,z^5,z^2$ will have a $0$ appended, $z^3,z^4$ will have a $1$, and $z^1,z^7$ will have a $2$ appended:
$$
Z^2 = (21,01,10,12,00,00,21).
$$
For the last permutation, we have $\sigma^3=1423567,$ and
$$
\Des_{R,((21) , (01),(10),(12),(00),(00),(21))}   \left( 1423567  \right) = \{ 1, 2, 4,5  \}.
$$
The resulting sequences would be given by
$$
Z^3 = (021,201,210,112,300,400,421).
$$
The last step is to add the last coordinate to get $S= (s^1,\dots,s^n)$ with $s^i + \chi(i \in R) \leq s^{i+1}.$ 
Since  $\Des_{R,Z^3}( \epsilon)  = \{3\},$
we get
\begin{align*}
\begin{tikzpicture}[scale= .7]
\draw (-2,0) node {$Z^{4}=$};
\draw (0,0) node {$0021$};
\draw (1,0) node {$\leq$};
\draw (2,0) node {$0 201  $};
\draw (3,0) node {$<$};
\draw (4,0) node {$0 210$};
\draw (5,0) node {$\leq$ };
\draw (6,0) node {$1112$};
\draw (7,0) node {$\leq$};
\draw (8,0) node {$1300$};
\draw (9,0) node {$ <$};
\draw (10,0) node {$1400 $};
\draw (11,0) node {$< $};
\draw (12,0) node {$1421  $};
\end{tikzpicture}
\end{align*}

The point of this construction is that 
$Z^4$ will correspond to the following semistandard tableau:
\begin{align*} 
\begin{tikzpicture}[scale= .85]
\partitionfr{4,2,1}\draw (-2,1.5) node {$P~=$};
\draw (.5,.5) node {$0021$};
\draw (1.5,.5) node {$ 0201 $};
\draw ( .5,1.5) node {$0210$};
\draw (2.5,.5) node {$1112 $};
\draw (3.5,.5) node {$1300$};
\draw (1.5,1.5) node {$1400$};
\draw (.5,2.5) node {$1421$};
\end{tikzpicture}
\end{align*}
For $S \in (\mathbb{N}^r)^n$ let 
$$
q^S = \prod_{i=1}^r \prod_{j=1}^n q_i^{s^j_{r-i+1}}
$$
Then we have 
$$
\prod_{i=1}^k q_i^{\comaj^i_T (\vec{\sigma})} = q^{Z^k_R(\vec{\sigma})} = q_1^5 q_2^6  q_3^{16} q_4^4
$$ 
To get from Theorem \ref{finiteevaluation} to Theorem \ref{KroneckerHarmonics}, we set 
$$
\comaj_T(\vec{\sigma},\epsilon) = \comaj^1_T(\vec{\sigma}) + \cdots + \comaj_T^k(\vec{\sigma}) 
$$
In this example, we would have
$$
\comaj_T(3651274, 6523417, 1423567,\epsilon) = |P| = 5+6+16+4 = 31,
$$
where $|P|$ is the sum of all the entries in the tableau $P$.
This connection will be made clearer in the following sections.

\section{Proving the principal evalutations}
\subsection{A general proposition}
For any $R \subseteq \{1,\dots, n-1\}$ and list $S = (s^1,\dots, s^n) \in (\mathbb{N}^r)^n$, we will need to define a permutation $\sigma_R(S) \in S_n$ which reads the sequences in $S$ in a particular nondecreasing order.  To get this permutation, we read $s^1,\dots, s^n$ in increasing order, where if $s^i = s^j$, we read $s^i$ before $s^j$  whenever
\begin{itemize}
\item $i<j$ and $i \not \sim_R j$, or
\item $i> j$ and $i \sim_R j$.
\end{itemize}
We call this the reading order of $S$. \\

For example, for $R = \{2,3,4,7\}$ and 
$$
S = (110,010,010,110,010,210,110,210,010),
$$
we have
$$\sigma_R(S) = 532914768.
$$

Recall that for any list of sequences $Z$, we set
$$
q^Z = \prod_{i=1}^r q_{r-i+1}^{z^1_i+\cdots + z^n_i}  = q_1^{z^1_r+\cdots +z^n_r} \cdots q_r^{z^1_1+ \cdots +z^n_1}.
$$ 
Also, let $Z^{(i)} = (s^1,\dots, s^n)$ with $s^j=(z^j_i,\dots, z^j_n)$.
Then

\begin{proposition} \label{generalproposition} For any two subsets $R,D \subseteq \{1,\dots, n-1\}$, and any permutation $\sigma \in S_n$,
\begin{align*}
(q_r;q_r)_n \sum_{ \substack { Z \in (\mathbb{N}^r)^n \\ \sigma_R(Z) = \sigma   }} q^Z \chi \left(\Des_{R,Z^{(2)}}(\sigma) =D \right)   & =   q_r^{c(D)} \sum_{ \substack{ S \in (\mathbb{N}^{r-1})^n \\ 
    \Des_{R,S}(\sigma)  = D    }}  q^S,
\end{align*}
where $c(D) = \sum_{i \in D} (n-i) = \Comaj_{R,S}(\sigma).$
\end{proposition}
\begin{proof}
Define a family of injections $\{\phi_{n-i}\}$ on the set of $Z \in( \mathbb{N}^r)^n$ with $\sigma_R(Z) = \sigma$ by setting $\phi_{n-i}(Z) = S=(s^1,\dots, s^n) \in (\mathbb{N}^r)^n$ with $Z^{(2)} = S^{(2)}$ and
\begin{enumerate}
\item $s^{\sigma_j}_1 = z^{\sigma_j}$ for $j \leq i$, and 
\item $s^{\sigma_j}_1 = z^{\sigma_j} +1$ for $j > i$.
\end{enumerate}
In other words, we get $S$ from $Z$ by adding $1$ to all the first coordinates in $z^{\sigma_j}$ for $j>i$. 
This means that $q_r^{n-i} q^Z = q^{\phi_{n-i}(Z)}$.
By the definitions of $\phi_{n-i}$ and $\sigma_R(Z)$, we have 
$$
\sigma_R(Z) = \sigma_R( \phi_{n-i} Z).
 $$
and $(\phi_{n-i} Z )^{(2)} = Z^{(2)},
$
meaning
$$
\Des_{R,Z^{(2)}}(\sigma) = \Des_{R, (\phi_{n-i} Z)^{(2)}}(\sigma).
$$
For example, if we apply $\phi_{7-3}$ to $(01,20,11,20,11,20,01)$ when $R = \{3,4,5\}$, we would get
$$
(01,30,21,30,11,30,01).
$$
Note that the $11$ in position $5$ is preserved by $\phi_{7-3}$ since it is read before the $11$ in position $3$. This is because $3 \sim_R 5$.
\\

We now begin by looking at
$$
(1-q_r^n) \sum_{ \substack { Z \in (\mathbb{N}^r)^n \\ \sigma_R(Z) = \sigma   }} q^Z \chi(\Des_{R,Z^{(2)}}(\sigma) =D) .
$$
Let 
$$C = \{ Z \in (\mathbb{N}^r)^n: \text{ $\sigma_R(Z) = \sigma$ and $\Des_{R,Z^{(2)}}(\sigma) =D$}\}
$$
Define $C^n = C - \phi_nC$ to be the set $C$ with the image of $\phi_n$ removed. Note that $\phi_n$ is reversible if all the first entries are nonzero. Therefore, we have
$$
C^n = \{Z \in C: z^{\sigma_1}_1 = 0\}, \text{ and}
$$
$$
(1-q_r^n) \sum_{ Z \in C   } q^Z
= \sum_{ { Z \in C^n }} q^Z.
$$
We now recursively define
$$
C^{n-i} = C^{n-i+1} - \phi_{n-i} C^{n-i+1}.
$$
Suppose we have $S \in C^{n-i+1}$ and suppose, by induction, that we have for $j<i$
$$
s^{\sigma_{j+1}}_1 =s^{\sigma_{j}}_1+ \chi(j \in D).
$$
Then $\phi_{n-i} C^{n-i+1} \subseteq C^{n-i+1}$, since $\phi_{n-i}$ acts only on the last $n-i$ coordinates, while $C^{n-i+1}$ is defined by the relations in the first $i$ coordinates. Now observe that we have
$$
S \in C^{n-i} =  C^{n-i+1}-\phi_{n-i} C^{n-i+1}
$$
if we cannot remove $1$ from all the first entries in $s^{\sigma_j}$ for $j > i$. We need only check that we cannot remove $1$ from $s^{\sigma_{i+1}}_1$ in replacing $s^{\sigma_{i+1}}$ by $z = (s^{\sigma_{i+1}}_1-1,s^{\sigma_{i+1}}_2,\dots, s^{\sigma_{i+1}}_n)$. By this, we mean that $z$ is either lexicographically smaller than $s^{\sigma_i} $ or $z = s^{\sigma_i}$ and $z$ occurs before $s^{\sigma_i}$ in the reading order. The main observation is that we cannot remove a $1$ from $s^{\sigma_{i+1}}_1$ precisely when 
$$
 s^{\sigma_{i+1}}_1 =  s^{\sigma_{i }}_1 + \chi(i \in D).
$$
If $s^{\sigma_{i+1}}_1 >  s^{\sigma_{i }}_1+1$, then $z$ is lexicographically larger than $s^{\sigma_i}$. If $ s^{\sigma_{i+1}}_1 =  s^{\sigma_{i }}_1 $, then $z < s^{\sigma_i}$.
We are left to consider the case when  $s^{\sigma_{i+1}}_1 =  s^{\sigma_{i }}_1 + 1$; we must show that we cannot remove $1$ from $s^{\sigma_{i+1}}_1$ precisely when $i \in D$. To simplify notation, let $\overline{z} = (s^{\sigma_{i+1}}_2,\dots, s^{\sigma_{i+1}}_n)$ and $\overline{s} = (s^{\sigma_{i }}_2,\dots, s^{\sigma_{i }}_n)$. We need to study four cases:
\begin{enumerate}
\item If $\sigma_{i+1} < \sigma_i$, and $\sigma_{i+1} \sim_R \sigma_i$
then we would have $z$ before $s^{\sigma_i}$ in the reading order whenever 
$$
 \overline{s} > \overline{z}  .
$$

\item If $\sigma_{i+1} < \sigma_i $, and $\sigma_{i+1} \not\sim_R \sigma_i$,
then $z$ will be read before $s^{\sigma_i}$ if
$$
\overline{s} \geq \overline{z} . 
$$
\item If $\sigma_{i+1} > \sigma_i$, and $\sigma_{i+1} \sim_R \sigma_i$
then we would have $z$ before $s^{\sigma_i}$ in the reading order whenever 
$$
 \overline{s} \geq \overline{z}  .
$$

\item If $\sigma_{i+1} > \sigma_i $, and $\sigma_{i+1} \not\sim_R \sigma_i$,
then $z$ will be read before $s^{\sigma_i}$ if
$$
\overline{s} > \overline{z} . 
$$
\end{enumerate}
This means that $z$ would be read before $s$ if 
\begin{align*}
\overline{s} > \overline{z} &-\chi( \text{$\sigma_{i+1} < \sigma_i$ and $\sigma_{i+1} \not \sim_R \sigma_i$})\\
&-\chi( \text{$\sigma_{i+1} > \sigma_i$ and $\sigma_{i+1} \sim_R \sigma_i$}).
\end{align*}
This is precisely when we cannot remove a $1$ from the first coordinate of $s^{\sigma_{i+1}}$, meaning, we cannot remove a $1$ when 
 $$
 i \in \Des_{R,S^{(2)}}(\sigma) = D.
 $$

By induction, we have shown that
$$
(q_r;q_r)_n \sum_{ \substack { Z \in (\mathbb{N}^r)^n \\ \sigma_R(Z) = \sigma   }} q^Z \chi(\Des_{R,Z^{(2)}}(\sigma) =D) 
= \sum_{  { Z \in C^1    }} q^Z,
$$
where $C^1$ consists of all lists of sequences $(s^1,\dots, s^n) \in ( \mathbb{N}^r)^n$ satisfying, for all $i<n$,
$$
s^{\sigma_{i+1}}_1 =s^{\sigma_{i}}_1+ \chi(i \in D),
$$
and $\Des_{R, S^{(2)} }(\sigma) = D$. The sum of the first coordinates in $s^1,\dots, s^n$ is equal to $c(D)$. Factoring $q_r^{c(D)}$ gives
$$
\sum_{ \substack{ Z \in C^1  \\ \sigma_R(Z) = \sigma  }} q^Z = q_r^{c(D)} \sum_{ \substack{ S \in (\mathbb{N}^{r-1})^n \\ 
    \Des_{R,S}(\sigma)  = D    }}  q^S.
$$
This completes the proof.
\end{proof}

This proposition gives us the following application:
\begin{corollary}  For any $R \subseteq \{1,\dots,n-1\},$
\begin{align*}
(q_r;q_r)_n \sum_{ \substack { Z \in (\mathbb{N}^r)^n  \\\sigma_R(Z) = \sigma   }} q^Z   & =   \sum_{ { S \in (\mathbb{N}^{r-1})^n   }} q_r^{\Comaj_{R,S}(\sigma)} q^S \\
& = \sum_{ { S \in (\mathbb{N}^{r-1})^n   }}  q^{Z_{R,\sigma}(S)}.
\end{align*}
 
\end{corollary}

\subsection{Applying the proposition}

We know that 
$$
s_{\lambda} \left[ \frac{1}{(1-q_1)\cdots (1-q_k) }\right] = \sum_{P \in \SSYT(\lambda, \mathbb{N}^k)} q^P,
$$
where if $P$ is a semistandard tableau with entries $(s^1,\dots, s^n) \in (\mathbb{N}^k)^n$ ordered lexicographically, we set
$$
q^P = q^S = q_1^{s^1_k+\cdots +s^n_k} \cdots q_k^{s^1_1+ \cdots +s^n_1}.
$$
A semistandard Young tableau of shape $\lambda$ with entries given by $S$ standardizes to a standard tableau $T$ if 
$$
s^i \leq s^{i+1} - \chi(i \in \des(T)).
$$
On the other hand, every semistandard tableau that standardizes to $T$ is given by sequences of this form. Therefore,
$$
\sum_{P \in \SSYT(\lambda, \mathbb{N}^k)} q^P = \sum_{T\in \SYT(\lambda)} \sum_{ \substack{  S \in (\mathbb{N}^k)^n      \\     \sigma_{\des(T)}(S) = \epsilon    }   } q^S.
$$
To simplify our expressions, let $R = \des(T)$. Then
applying the lemma to the right-most sum gives
\begin{align*} 
(q_k;q_k)_n  \sum_{ \substack{  S^k  \in (\mathbb{N}^k)^n      \\     \sigma_{R}(S^k) = \epsilon    }   } q^{S^k} 
& = \sum_{D^k \subseteq \{1,\dots, n-1\} } q_k^{c(D^k)} \sum_{\sigma \in S_n} \sum_{ \substack{  S^{k-1} \in  (\mathbb{N}^{k-1})^n      \\     \sigma_{R}(S^{k-1}) = \sigma    }   }  q^{S^{k-1}} \chi(\Des_{R,S^{k-1}}(\epsilon) = D^k).
\end{align*}
We now multiply the result by $(q_{k-1};q_{k-1})_n$ and relabel $\sigma$ to be $\sigma^{k-1}$.
\begin{align*} 
&(q_{k-1} ;q_{k-1})_n  (q_{k } ;q_{k })_n  
\sum_{ \substack{  S^k \in (\mathbb{N}^k)^n      \\     \sigma_{R}(S^k) = \epsilon    }   } q^{S^k} \\
 & = \sum_{D^k \subseteq \{1,\dots, n-1\} } q_k^{c(D^k)} \sum_{\sigma^{k-1} \in S_n} ~ (q_{k-1};q_{k-1})_n  \sum_{ \substack{  S^{k-1} \in  (\mathbb{N}^{k-1})^n      \\     \sigma_{R}(S^{k-1}) = \sigma^{k-1}    }   }  q^{S^{k-1}} \chi(\Des_{R,S^{k-1}}(\epsilon) = D^k)
\end{align*}
 Now apply Proposition \ref{generalproposition} on the last summation to get 
 \begin{align*}
  (&q_{k-1} ;q_{k-1})_n    \sum_{ \substack{  S^{k-1} \in  (\mathbb{N}^{k-1})^n      \\     \sigma_{R}(S^{k-1}) = \sigma^{k-1}    }   }  q^{S^{k-1}}    \chi(\Des_{R,S^{k-1}}(\epsilon) = D^k) \\
& =\sum_{D^{k-1} \subseteq \{1,\dots, n\} }   q_{k-1}^{c(D^{k-1})}   \sum_{\sigma^{k-2} \in S_n} \sum_{ \substack{  S^{k-2} \in  (\mathbb{N}^{k-2})^n      \\     \sigma_{R}(S^{k-2}) = \sigma^{k-2}    }   }  q^{S^{k-2}} \\  & ~~~~
\times \chi( \Des_{R,   Z_{R, \sigma^{k-1}}(S^{k-2}) }  (\epsilon) = D^k    )   \chi(\Des_{R, S^{k-2}}(\sigma^{k-1}) =D^{k-1} ).
\end{align*}

We must make an important comment here. Note that there is the extra factor of $ \chi(\Des_{R,S^{k-1}}(\epsilon) = D^k) $ which does not appear in the proposition. 
However, this factor can be ignored and carried over as we did in the above computation.
To see this, we must make sure that the maps $\{\phi_{n-i}\}$ defined in Proposition \ref{generalproposition} 
have the following property:

\begin{lemma}
If for $S \in (\mathbb{N}^m)^n$ we have $\Des_{R,Z_{R,(\sigma^1,\dots,\sigma^r)} S}(\tau) = D$, then for any injection $\phi_{n-j}$ defined in Proposition \ref{generalproposition}, we have
$$
\Des_{R,Z_{R,(\sigma^1,\dots,\sigma^r)} \phi_{n-j} S}(\tau) = D
$$
\end{lemma}
Before beginning the proof, we look at a quick example. Suppose $R= \{3,4,5\}$ and we are given 
$$
Z = (2001, 1020,0111,1020,0111,1020,2001).
$$
Then $\Des_{R,Z}(5362471) = \{3,6\}.$ Now we are going to apply $\phi_{7-3}$ to 
$$Z^{(3)} = (01,20,11,20,11,20,01).$$
As we saw in the proof of Proposition \ref{generalproposition}, we get
$$
\phi_{4}Z^{(3)} = (01,30,21,30,11,30,01).
$$
Therefore, applying $\phi_4$ to the third coordinates in $Z$, we get
$$
Z' = (2001,1030,0121,1030,0111,1030,2001).
$$
Then the main observation is that we have $\Des_{R,Z'}(5362471) = \{3,6\} = \Des_{R,Z}(5362471).$

\begin{proof}
Let us denote $Z_{R,(\sigma^1,\dots,\sigma^r)} S$ by $Z= (z^1,\dots, z^n)$, and $Z_{R,(\sigma^1,\dots,\sigma^r)} \phi_{n-j} S$ by $W=(w^1,\dots, w^n)$. Note that this means that 
$$
z^i = (z^i_1,\dots, z^i_r,s^i_1,\dots,s^i_{m}).
$$

Suppose $i \in \Des_{R,Z}(\tau) = D$. We want to show that $i \in \Des_{R,W}(\tau) $. By definition, we know we must have
\begin{align*}
z^{\tau_i} > z^{\tau_{i+1}} 
&- \chi(   \text{ $\tau_{i+1} < \tau_i $ and $\tau_{i+1} \not\sim_R\tau_{i}    $} )  \\ 
&- \chi( \text{ $\tau_{i+1} > \tau_i $ and $\tau_i \sim_R\tau_{i+1}  $})               \}.
\end{align*}
If $z^{\tau_i} > z^{\tau_{i+1}}$, then we either have  
$$
(z^{\tau_i}_1,\dots, z^{\tau_i}_r) > (z^{\tau_{i+1}}_1,\dots, z^{\tau_{i+1}}_r) 
$$
or
$$
\text{$ (z^{\tau_i}_1,\dots, z^{\tau_i}_r) = (z^{\tau_{i+1}}_1,\dots, z^{\tau_{i+1}}_r) $ and $s^{\tau_i} > s^{\tau_{i+1}}$}.
$$
In the first case, no matter if $\phi_{n-j}(S)$ increases the $\tau_i$ or $\tau_{i+1}$ sequence in $S$, we will always have 
$$
w^{\tau_{i}}>w^{\tau_{i+1}}.
$$
In the second case, since $s^{\tau_i} > s^{\tau_{i+1}}$, $\phi_{n-j}$ will either increase the first coordinates of both, or only the first coordinate of $s^{\tau_i}$. In either case, we have 
$
w^{\tau_{i}}>w^{\tau{i+1}}.
$
\\

Now if $z^{\tau_i} = z^{\tau_{i+1}}$ then we must have that either $\tau_{i+1} < \tau_i $ and $\tau_{i+1} \not\sim_R\tau_{i}    $; or $\tau_{i+1} > \tau_i $ and $\tau_i \sim_R\tau_{i+1}  $. In both cases we have that $s^{\tau_i} = s^{\tau_{i+1}}.$ The first case implies that $\phi_{n-j}$ will either increase both sequences, or only $s^{\tau_{i}}$. This is because $s^{\tau_{i+1}}$ appears before $s^{\tau_i}$ and $\tau_{i+1} \not \sim_R \tau_i$. In the second case, we have that $\tau_i \sim_R \tau_{i+1}$, but $\tau_{i}$ appears first. This means that $\phi_{n-j}$ will either increase both or only $s^{\tau_i}$. In either case, we get that $w^{\tau_i} \geq w^{\tau_{i+1}}$. In all cases above, we get that $i \in \Des_{R,W}(\tau)$. 
\\

Now suppose instead that $i \notin D$. This means that 
\begin{align*}
z^{\tau_i} \leq z^{\tau_{i+1}} 
&- \chi(   \text{ $\tau_{i+1} < \tau_i $ and $\tau_{i+1} \not\sim_R\tau_{i}    $} )  \\ 
&- \chi( \text{ $\tau_{i+1} > \tau_i $ and $\tau_i \sim_R\tau_{i+1}  $})               \}.
\end{align*}
If $z^{\tau_i} < z^{\tau_{i+1}}$, then again we either have 
$$
(z^{\tau_i}_1,\dots, z^{\tau_i}_r) < (z^{\tau_{i+1}}_1,\dots, z^{\tau_{i+1}}_r) 
$$
or
$$
\text{$ (z^{\tau_i}_1,\dots, z^{\tau_i}_r) = (z^{\tau_{i+1}}_1,\dots, z^{\tau_{i+1}}_r) $ and $s^{\tau_i} < s^{\tau_{i+1}}$}.
$$
In the first case, $\phi_{n-j}$ does not affect the order of the first $r$ coordinates, meaning $w^{\tau_i} < w^{\tau_{i+1}}$. In the second case, since $s^{\tau_i} <s^{\tau_{i+1}}$, we have that $\phi_{n-j}$ will either increase both sequences or only $s^{\tau_{i+1}}$. In either case, we get $w^{\tau_i} < w^{\tau_{i+1}}$. 
\\

If instead we have that $z^{\tau_i} = z^{\tau_{i+1}}$ (and therefore $s^{\tau_i} = s^{\tau_{i+1}}$), then we either have $\tau_{i+1} < \tau_i $ and $\tau_{i+1} \sim_R\tau_{i}    $; or $\tau_{i+1} > \tau_i $ and $\tau_i \not \sim_R\tau_{i+1}  $. In the first case, since $\tau_{i} \sim_R \tau_{i+1}$, $\phi_{n-j}$ will either increase both $s^{\tau_i}$ and $s^{\tau_{i+1}}$ or only $s^{\tau_{i+1}}$. In the second case, since $\tau_{i} \not \sim_R \tau_{i+1}$ but this time $\tau_{i+1} > \tau_i$, we have that $\phi_{n-j}$ will either increase both $s^{\tau_i}$ and $s^{\tau_{i+1}}$ or only $s^{\tau_{i+1}}$. In either case, we end with $w^{\tau_i} \leq w^{\tau_{i+1}}$. From looking at all cases, we get that $i \notin \Des_{R,W}(\tau)$. This concludes the proof.

\end{proof}

This Lemma allows us to continue applying Proposition \ref{generalproposition} to all coordinates. Continuing the previous computation, we end with
\begin{align*}
\left( \prod_{i=1}^k (q_i;q_i)_n \right)  &\sum_{ \substack{  Z \in (\mathbb{N}^k)^n      \\     \sigma_{\des(T)}(Z) = \epsilon    }   } q^Z  \\
& = \sum_{D^k,\dots, D^1} \sum_{\vec{\sigma} \in S_n^{\times k-1}}
 q_k^{c(D^k)} \cdots q_1^{c(D^1)}\prod_{i=1}^k \chi(\Des_{R,Z^{i-1}}(\sigma^{i}) =D^{i} ).
\end{align*}
where $Z^i = Z_{R,\sigma^i} (Z^{i-1})$ with $Z^0 = \emptyset$, and $\sigma^k = \epsilon$.
Note that for each $\sigma^i$ there is precisely one $D^i$ for which $(\Des_{R,Z^{i-1}}(\sigma^{i}) =D^{i})$ holds. The surviving terms are then determined by the vector $\vec{\sigma}\in S_n^{\times k-1}$ and the lists 
$$
Z^i = Z_{R,\sigma^i} \cdots Z_{R,\sigma^1} (\emptyset) = Z^i_{R,\vec{\sigma}}.
$$
Furthermore, 
$$c(D^i) = \comaj^i_R(\vec{\sigma}) = \Comaj_{R,Z^{i-1}}(\sigma^i).$$
Putting everything together, we get
\begin{align*}
\left( \prod_{i=1}^k (q_i;q_i)_n \right)  \sum_{ \substack{  Z \in (\mathbb{N}^k)^n      \\     \sigma_{R}(Z) = \epsilon    }   } q^Z 
&= \sum_{ \vec{\sigma} \in S_n^{\times k-1} }    q^{Z^{k}_{R,\vec{\sigma}} }       \\
& = \sum_{ \vec{\sigma} \in S_n^{\times k-1} } q_1^{\comaj^1_R(\vec{\sigma})} \cdots q_k^{\comaj^k_R(\vec{\sigma})} 
\end{align*}
This completes the proof of Theorem \ref{finiteevaluation}, but it also says a little more. Recall that we started with a standard tableau $T$ whose descent set is $R$. 
Let $Z = (z^1,\dots, z^n) = Z^k_{R,\vec{\sigma}} $  be as defined above, and let $P_{T,\vec{\sigma}} $ be the semistandard tableau one gets by replacing $i$ by $z^i$ in $T$.
Define
$$
\mathcal{T}^k_{\lambda} = \{ P_{T,\vec{\sigma}}: \text{ $T\in \SYT(\lambda)$ and $\vec{\sigma } \in S_n^{\times k-1}$} \}.
$$
For instance, we saw in the previous section that for 
\begin{align*} 
\begin{tikzpicture}[scale= .5]
\partitionfr{4,2,1}\draw (-2,1) node {$T~=$};
\draw (.5,.5) node {$1$};
\draw (1.5,.5) node {$ 2 $};
\draw ( .5,1.5) node {$3 $};
\draw (2.5,.5) node {$4 $};
\draw (3.5,.5) node {$5$};
\draw (1.5,1.5) node {$6$};
\draw (.5,2.5) node {$7$};
\end{tikzpicture},
\end{align*}
\begin{align*} 
\begin{tikzpicture}[scale= .85]
\partitionfr{4,2,1}\draw (-3,1.5) node {$P_{T, (3651274, 6523417, 1423567)}~=$};
\draw (.5,.5) node {$0021$};
\draw (1.5,.5) node {$ 0201 $};
\draw ( .5,1.5) node {$0210$};
\draw (2.5,.5) node {$1112 $};
\draw (3.5,.5) node {$1300$};
\draw (1.5,1.5) node {$1400$};
\draw (.5,2.5) node {$1421$};
\end{tikzpicture}
\end{align*}
is an element of $\mathcal{T}^4_{(4,2,1)}$.

\begin{theorem}
For any $n$, $k$, and $\lambda \vdash n$, we have
$$
\left( \prod_{i=1}^k (q_i;q_i)_n \right) s_{\lambda} \left[ \frac{1}{(1-q_1)\cdots(1-q_k) } \right] = \sum_{P \in \mathcal{T}^k_\lambda} q^P.
$$
\end{theorem}

We will now work out a new example. Suppose 
\begin{align*} 
\begin{tikzpicture}[scale= .5]
\partitionfr{2,2,1,1}\draw (-2,1) node {$T~=$};
\draw (.5,.5) node {$1$};
\draw (1.5,.5) node {$ 3$};
\draw ( 1.5,1.5) node {$4 $};
\draw (.5,1.5) node {$2 $};
\draw (.5,2.5) node {$5$};
\draw (.5,3.5) node {$6$};
\end{tikzpicture},
\end{align*}
Then $R= \{1,3,4,5\}$. This says that $1 \sim_R 2$ and $3 \sim_R 4 \sim_R 5 \sim_R 6$. We will now construct
$Z_{R,(\vec{\sigma})}^2$ for $\vec{\sigma} = ( 631254, 365412)$.
First note that $\Des_{R,\emptyset}(631254) = \{ 2, 3\}$ giving
$$
Z^1 = (1 ,2 ,0 ,2 ,2 , 0).
$$
Next, $\Des_{R,Z^1}(365412) = \{  1, 4   \}$, giving
$$
Z^2 = (21 ,22 , 00 ,12 ,12 , 10).
$$
Lastly, we look at $\Des_{R,Z^2}(\epsilon) = \{  2,4,5   \}$ to get
$$
Z^3 =  (021 ,022 , 100 ,112 ,212 , 310).
$$
This gives

\begin{align*} 
\begin{tikzpicture}[scale= .85]
\partitionfr{2,2,1,1}\draw (-3,1.5) node {$P_{T, (  631254, 365412  )}~=$};
\draw (.5,.5) node {$021$};
\draw (.5,1.5) node {$ 022 $};
\draw ( 1.5,.5) node {$100$};
\draw (1.5,1.5) node {$112 $};
\draw (.5,2.5) node {$212$};
\draw (.5,3.5) node {$310$};
\end{tikzpicture}
\end{align*}

\section{The infinite evaluation}
To take $k \rightarrow \infty$ we first define $\mathbb{N}^{\infty}$ to be the set of sequences $(s_1,s_2,\dots)$ with entries in $\mathbb{N}$, with the extra condition that for some $m$ we have $s_i = 0$ for all $i>m$. Then we have
$$
s_{\lambda}\left[ \frac{1}{(1-q_1)(1-q_2) \cdots} \right] = \sum_{T \in \SYT(\lambda)} \sum_{ \substack{ Z \in (\mathbb{N}^{\infty})^n
\\ \sigma_{\des(T)} (Z) = \epsilon}} Q^Z,
$$
where we now set
$$
Q^Z = \prod_{i \geq 1} q_i^{z^1_i+ \cdots z^n_i}.
$$
If $k$ is the last index in $Z=(z^1,\dots, z^n)$ for which $z^i_k \neq 0$ for some $k$, we have 
$$
Q^Z = q^{Z'} \Big|_{q_i \rightarrow q_{k-i+1}},
$$
where $Z'$ is the list $Z$ with all sequences restricted to the first $k$ coordinates.
Let $Z^{(i)} = (s^1,\dots, s^n)$, with $s^j  = (z^j_i,z^j_{i+1} ,\dots) $,
be the list $Z$ with the first $i-1$ coordinates of each sequence $z^j$ removed. 
We can still apply Proposition \ref{generalproposition} since it depends only on the first coordinates of our list $Z$. 
Eventually, however, we must reach an index $i$ for which 
$$
Z^{(i)} = ((0,\dots),\dots, (0,\dots))
$$
consists purely of zero entries. Given any set $R \subseteq\{1,\dots,n-1\}$, let $\tau_R$ be the unique permutation in $S_n$ for which $\Comaj_{R,\emptyset}(\tau_R) =0.$
The sequence $(\sigma_R(Z^{(j)}))_{j\geq 1} = (\sigma^1,\sigma^2,\dots)$ must then satisfy the following conditions:
\begin{enumerate}
\item $\sigma^1 = \epsilon,$ and
\item there is an $m$ for which $\sigma^i= \tau_R$ for $ i > m$.
\end{enumerate}
Define $_RS_n^{\infty}$ be the set of all such infinite sequences of permutations. 
For such a sequence $\vec{\sigma}$, there is a smallest $m$ for which $\sigma^m \neq \tau_R$ but $\sigma^i = \tau_R$ for $i>m$. Let $\underline{\vec{\sigma}} = (\sigma^m,\dots, \sigma^2).$ We then set
$$
Q^{\Comaj_R(\vec{\sigma})} = q_1^{\comaj^m_R(\underline{\vec{\sigma}})} 
\cdots q_m^{\comaj^1_R(\underline{\vec{\sigma}})}.
$$
We have
$$
\left( \prod_{i \geq 1}  (q_i;q_i)_n \right)  \sum_{ \substack{ \vec{\sigma} \in (\mathbb{N}^{\infty})^n
\\ \sigma_{R} (S) = \epsilon}} Q^Z
=\sum_{ \vec{\sigma} \in _RS_n^{\infty}} Q^{\Comaj_R(\vec{\sigma})},
$$
proving Theorem \ref{infiniteevaluation}.

\subsection{The quasisymmetric case}

We have actually shown something stronger. Recall Gessel's Fundamental basis for quasisymmetric functions \cite{Gessel}: For any subset $R \subseteq \{1,\dots,n-1\}$, we have
\begin{align*}
F_{n,R}(x_1,x_2,\dots) & = \sum_{\substack{ a_1\leq \cdots \leq a_n  \\  a_i < a_{i+1}  \text{ for $ i \in R$}}} x_{a_1} \cdots x_{a_n} .
\end{align*}
Now, letting $1/(1-q_1)(1-q_2) \cdots$ denote all monomials $1,q_1, q_2\dots,q_1^2,q_1q_2,\dots$ in the variables $q_1,q_2,\dots$ we can write
$$
F_{n,R} \left[ \frac{1}{(1-q_1)\cdots(1-q_k) }\right] = \sum_{ \substack{  S \in (\mathbb{N}^k)^n  \\  \sigma_R(S) = \epsilon  }} q^S.
$$
The condition $\sigma_R(S) = \epsilon$ is another way of writing the inequality conditions in the definition of the fundamental basis, where now, instead of looking at integers, we are looking at inequalities between sequences.
This is precisely the sum in Proposition \ref{generalproposition}. The computation in the proof of Theorem \ref{finiteevaluation} starts by summing over standard tableaux. We then isolated a descent set $R$ and made the following computation:
\begin{theorem} \label{quasitheorem}
For any $R \subseteq \{1,\dots, n-1\} $ and any $k$, we have
$$
F_{n,R} \left[ \frac{1}{(1-q_1)\cdots(1-q_k) }\right] = \frac{  \sum_{\vec{\sigma} \in S_n^{\times k-1}} q_1^{\comaj^1_R (\vec{\sigma})} \cdots q_k^{\comaj^k_R(\vec{\sigma}) }   }{(q_1;q_1)_n \cdots (q_k;q_k)_n}.
$$
In particular, letting $k$ go to infinity, we have
$$
F_{n,R} \left[ \frac{1}{(1-q_1) (1-q_2) \cdots}\right] =\frac{  \sum_{ \vec{\sigma} \in _RS_n^{\infty}} Q^{\Comaj_R(\vec{\sigma})} }{ (q_1;q_1)_n(q_2;q_2)_n\cdots}.
$$
\end{theorem}

\section{A special case}
We are going to look at the special case when $\lambda = (n)$ is a single row. There is only one semistandard tableau $T$ of shape $(n)$ and it has descent set $\emptyset$. To compute $\comaj_{T}(\vec{\sigma},\epsilon)$, we recall that
$
\comaj_{T}^r(\vec{\sigma})= \Comaj_{\emptyset, Z^{r-1}}(\sigma^r),
$
where 
$$
Z^{r-1} = Z_{\emptyset,\sigma^{r-1}}\cdots Z_{\emptyset,\sigma^1}(\emptyset).
$$
The simplest case to start with is when $k = 2$. Then our permutation vector consists of a single element $\sigma$. Constructing $Z_{\emptyset,\sigma} (\emptyset)$ is simply the labelling for $\comaj (\sigma)$ described in the introduction. 
\\

For instance, if $\sigma= 63482715,$ then we start with setting $s^{\sigma_1} = s^6 = (0)$. Since $\sigma_2 = 3 < 6 = \sigma_1$, we have a descent and we set $s^{\sigma_2}= s^3 = (1)$. Since $4 = \sigma_3 < \sigma_4 = 8$, we set $s^{\sigma_3} = s^{\sigma_4} = (1)$. We have a descent at $4$, so we set $s^{\sigma_5} = (2)$. In the end, we get
$$
( s^1,s^2,s^3,s^4,s^5,s^6,s^7,s^8) = ( (3) , (2)  ,(1), (1),   (3)   ,   (0)    , (2)  , (1)) .
$$
The sum of the entries is $\comaj(\sigma)$. Next, we add a new coordinate in front of each $s^{i}$. We start with  $z^1 = (0,s^1_1)$.  If we labeled $s^i$ by $c$, and $s^i_1 > s^{i+1}_1$, then we have a descent at $i$ and we label $s^{i+1}$ with $c+1$. Otherwise we label $s^{i+1}$ by $c$. Using the above example, we would have 
$$
( z^1,z^2,z^3,z^4,z^5,z^6,z^7,z^8) = ( (0,3) , (1,2)  ,(2,1), (2,1),   (2,3)   ,   (3,0)    , (3,2)  , (4,1)) .
$$
Note that $s^i_1 > s^{i+1}_1$ means that $\sigma^{-1}_i > \sigma^{-1}_{i+1},$ or rather $i \in \des(\sigma^{-1})$. Therefore, the sum of the first entries is $\comaj(\sigma^{-1})$. We get

\begin{corollary}
For any $n$, we have 
$$
(q_1;q_1)_n (q_2,q_2)_n h_n \left[ \frac{1}{(1-q_1)(1-q_2)} \right] = \sum_{\sigma \in S_n} q_1^{\comaj(\sigma^{-1})} q_2^{\comaj(\sigma)}. 
$$
As a consequence, we have
$$
\langle \mathcal{F} \bH_n^{\otimes 2}, s_n \rangle = \sum_{\sigma \in S_n} q^{\comaj(\sigma)+ \comaj(\sigma^{-1})}.
$$
\end{corollary}
The following Lemma will give us the general case.
\begin{lemma}
For any $\sigma \in S_n$ and $S \in (\mathbb{N}^r)^n$ with $\sigma_{\emptyset}(S) = \tau$, we have
$$
\Des_{\emptyset,S}(\sigma) = \des(\tau^{-1} \sigma).
$$
\end{lemma}
\begin{proof}
 We know that since $R$ is empty in the definition of descents, we have
\begin{align*}
D=\Des_{\emptyset,S}(\sigma) = \{i :  s^{\sigma_{i}} >  s^{\sigma_{i+1}}- \chi(\sigma_{i+1} < \sigma_i)                        \}
\end{align*}
Let $i \in D$. If $\sigma_{i+1} > \sigma_i$, then we are guaranteed that $s^{\sigma_{i}} > s^{\sigma_{i+1}}. $ By the definition of $\sigma_R(S) = \tau$, we have that $\tau^{-1}(\sigma_i) >\tau^{-1}(\sigma_{i+1}).$
In other words, 
$$
i \in \des(\tau^{-1}\sigma).
$$
If $\sigma_{i+1} < \sigma_i$, then we have a descent at $i$ provided  $s^{\sigma_i} \geq s^{\sigma_{i+1}}$. If the inequality is strict, we see as before that $i \in \des(\tau^{-1} \sigma)$. Suppose instead that $s^{\sigma_i} = s^{\sigma_{i+1}}$. Then since $\tau$ must read $ s^{\sigma_{i+1}}$ before $s^{\sigma_i}$ (since $s^{\sigma_{i+1}}$ is on the left of $s^{\sigma_i}$), we have that 
$$
\tau^{-1}(\sigma_{i}) > \tau^{-1}(\sigma_{i+1}).
$$
Again, this means $i \in \des(\tau^{-1}\sigma).$  

Now suppose that $i \in \des(\tau^{-1} \sigma).$
Suppose first that $\sigma_{i+1} > \sigma_i$. Then since we know
$$
\tau^{-1}(\sigma_{i}) > \tau^{-1}(\sigma_{i+1}),
$$
we have that $s^{\sigma_i} > s^{\sigma_{i+1}}$. This is because $s^{\sigma_i}$ lies on the left of $s^{\sigma_{i+1}}$, but $\tau$, which reads the $s^i$ in increasing order, reads  $s^{\sigma_{i+1}}$ before $s^{\sigma_i}$.
Now if $\sigma_{i+1} < \sigma_i$, then 
$\tau$ travels left to right from $s^{\sigma_{i+1}} $ to $s^{\sigma_i}$. This implies the inequality $s^{\sigma_{i+1}} \leq s^{\sigma_i}$. Combining both cases, we get
$$
s^{\sigma_{i}} >  s^{\sigma_{i+1}} - \chi(\sigma_{i+1} < \sigma_i),   
$$
or rather $i \in \Des_{\emptyset,S}(\sigma)$.
\end{proof}

The following theorem seems to be a variation of Garsia and Gessel's identities  \cite{GarsiaGessel}.
\begin{theorem}
\begin{align*}
h_n\left[ \frac{1}{(1-q_1)\cdots (1-q^k)} \right] = \frac{\sum\limits_{  \sigma^1\cdots \sigma^k = \epsilon  } q_1^{\comaj(\sigma^1 ) }   \cdots q_{k }^{\comaj(\sigma^{k }) }} { (q_1;q_1)_n\cdots (q_k;q_k)_n}
\end{align*}
\end{theorem}
\begin{proof}
Using Theorem \ref{finiteevaluation}, we need only check that when $T$ is a row of size $n$, we have
$$
\sum_{\vec{\pi} \in S_{n}^{\times k-1 } } q_1^{\comaj^1_T (\vec{\pi})} \cdots q_k^{\comaj^k_T(\vec{\pi}) } =
 \sum\limits_{  \sigma^1\cdots \sigma^k = \epsilon  } q_1^{\comaj(\sigma^1 ) }   \cdots q_{k }^{\comaj(\sigma^{k }) }
$$
By the previous Lemma, we have that for $1 \leq i \leq k$,
$$
\comaj^i_T(\pi^i ) = \comaj ( (\pi^{i-1})^{-1} \pi^i ),
$$
with $\pi^k = \pi^0 = \epsilon$.
Now let $\sigma^1 = \pi^1$, and set $\sigma^i = (\pi^{i-1})^{-1} \pi^i.$ With this reindexing, we have that 
$$
\comaj^i_T(\pi^i ) = \comaj (\sigma^i).
$$
The last part to note is that since $\pi^k = \epsilon,$ $\sigma^k$ is fully determined by the previous $\sigma^i$. Note that
$$
\sigma^1 \cdots \sigma^{k-1} \sigma^k = \epsilon \pi^1 (\pi^1)^{-1} \pi^2 (\pi^2)^{-1}\cdots \pi^{k-1} (\pi^{k-1})^{-1} \epsilon = \epsilon.
$$
This means $\sigma^k = (\sigma^1 \cdots \sigma^{k-1})^{-1}$. In turn, we have shown that 
$$
 q_1^{\comaj^1_T (\vec{\pi})} \cdots q_k^{\comaj^k_T(\vec{\pi}) } = q_1^{\comaj(\sigma^1 ) }   \cdots  q_{k }^{\comaj(  \sigma^{k})  }
$$
with $\sigma^1\cdots \sigma^k = \epsilon$,
completing the proof.
\end{proof}

\begin{corollary}
The Hilbert series for the $S_n$-invariants of $\bH_n^{\otimes k}$ is given by 
$$
\langle \mathcal{F} \bH_n^{\otimes k}, s_n \rangle = \sum\limits_{\sigma^1\cdots  \sigma^{k } = \epsilon } q^{\comaj(\sigma^1)+ \cdots + \comaj(\sigma^{k})   }.
$$
\end{corollary}

\section{Acknowledgements}
We must thank Nolan Wallach for suggesting the problem and for the numerous, helpful conversations. Also, thank you to Brendon Rhoades and Dun Qiu for the helpful comments and suggestions. The author was partially supported by the University of California President's Postdoctoral Fellowship.

\bibliographystyle{abbrv}
\bibliography{stabilizations}

{\small
  \noindent
  \\
  Marino Romero\\
  University of California, San Diego\\
  Department of Mathematics\\
  {\em E\--mail}: \texttt{mar007@ucsd.edu}
}

\end{document}